%% file: focal_hyp.tex
\documentclass[a4paper,10pt]{amsart}

\usepackage{amssymb,latexsym}

\usepackage{enumerate}

\makeatletter

\@namedef{subjclassname@2010}{%

  \textup{2010} Mathematics Subject Classification}

\makeatother

\numberwithin{equation}{section}

\begin{document}

\title[Focal Rigidity of Hyperbolic Surfaces]{Focal Rigidity of Hyperbolic Surfaces}

\author[F. H. Kwakkel]{Ferry H. Kwakkel}

\address{Instituto de Matematica Pura e Aplicada \\ Brazil}

\email{kwakkel@impa.br}

\date{}


\input{macros.tex}
\begin{abstract}
In this note, we consider the rigidity of the focal decomposition of closed hyperbolic surfaces. We show that, generically, the focal decomposition of a closed hyperbolic surface does not allow for non-trivial topological deformations, 
without changing the hyperbolic structure of the surface. By classical rigidity theory this is also true in dimension $n \geq 3$. Our current result extends a previous result that flat tori in dimension $n \geq 2$ that are focally equivalent are 
isometric modulo rescaling.
\end{abstract}

\subjclass[2010]{Primary 53C24; Secondary 53C22}

\keywords{focal decomposition, hyperbolic surfaces, rigidity}



\maketitle

\section{Definitions and Statement of Results}\label{sec_intro}

The purpose of this note is to consider the rigidity of the focal decomposition of closed hyperbolic surfaces. In order to state the precise result, we first recall the notion of focal decomposition and focal equivalence for general manifolds. Let $(M,g)$ be a closed, 
i.e. compact and boundaryless, and analytic ($C^{\omega}$) Riemannian $n$-manifold. A closed hyperbolic surface is naturally a closed analytic manifold in this sense. Fix a closed analytic manifold $(M,g)$. 

\begin{defn}\label{defn_index}
The focal index, $I(p,v)$, of the vector $v \in T_pM$ is defined by 
\begin{equation*}
I(p,v) = \# \left\{ w \in T_p M ~|~ |v| = |w| ~ \tu{and} \exp_p(v) = \exp_p(w) \right\}.
\end{equation*}
and we define 
\begin{equation*}
\sigma_i (p) = \left\{ v \in T_p M ~|~ i = I(p,v) \right\}.
\end{equation*}
\end{defn}

Thus, vectors $v \in \sigma_i(p)$ are equivalent modulo exponentiation to exactly $i-1$ other vectors of $T_p M$ of equal length. 

\begin{defn}[Focal decomposition]\label{defn_decomp}
The partition of $T_p M$ into the sets $\{ \sigma_i \}_{i=1}^{\infty}$ is called its {\em focal decomposition} at $p$; we have
\begin{equation}\label{eq_decomp}
T_p M = \bigcup_{i=1}^{\infty} \sigma_i ~\textup{and}~ \sigma_i \cap \sigma_j = \emptyset, ~\textup{if}~ i \neq j.
\end{equation}
The tangent bundle has a corresponding focal decomposition $\{ \Sigma_i\}_{ i=1}^{\infty}$, where 
\begin{equation}\label{eq_decomp_2}
\Sigma_i = \bigcup_{p \in M} \sigma_i (p) ~\tu{and}~TM = \bigcup_{i=1}^{\infty} \Sigma_i ~\textup{with}~ \Sigma_i \cap \Sigma_j = \emptyset ~\textup{if}~ i \neq j.  
\end{equation}
\end{defn}

In the setting of closed analytic manifolds the focal decomposition gives an analytic Whitney stratification of the tangent bundle of the manifold~\cite{KP}. This is sharp in the sense that the focal decomposition may be topologically pathological even for $C^{\infty}$ manifolds~\cite{pugh}. 
Further, for analytic manifolds, it follows from the Angle Lemma in~\cite{KP} that only $\sigma_1(p)$ can have interior. 

\begin{defn}[Focal equivalence]\label{focal_equiv}
Two closed analytic manifolds $(M_1, g_1)$ and $(M_2, g_2)$ are {\em focally equivalent}, if there exists an orientation-preserving homeomorphism $\varphi \colon TM_1 \ra TM_2$,
such that for every $p \in M_1$ and $q = \psi(p)$,
\begin{enumerate}
\item[\tu{(i)}] $\varphi_{|_{T_pM_1}} \colon T_pM_1 \ra T_{q} M_2$ and $\varphi_{|_{T_pM_1}}(0) = 0$,
\item[\tu{(ii)}] $\varphi_{|_{T_pM_1}} (\sigma^1_i(p)) = \sigma^2_i(q)$, for every $1 \leq i \leq \infty$,
\end{enumerate}
with $\psi \colon M_1 \ra M_2$ the homeomorphism on the zero section.
\end{defn}

It follows from (i) and (ii) in Definition~\ref{focal_equiv} that $\varphi(\Sigma_i^1) = \Sigma_i^2$. Note that we do not require the homeomorphism $\varphi$ to commute with the respective exponential mappings.
It is verified that focal equivalence indeed defines an equivalence relation. Further, manifolds that are isometric, up to rescaling, are focally equivalent. 
The focal decomposition gives in a natural way rise to Brillouin zones, which in a physical context arise in the theory of wave reflection by crystals on the quantum level. 

\begin{defn}[Brillouin zones]\label{defn_bril}
For $v \in T_p M$, we define the Brillouin index
\begin{equation}\label{eq_bril}
B(p,v) =  \# \left\{ w \in T_p M ~|~ |w | \leq |v|, ~ \exp_p(w) = \exp_p(v) \right\}.
\end{equation}
For every integer $k \geq 1$, the $k$-th {\em Brillouin zone} is the interior $\Int(B_k(p))$, of the set 
\[ B_k(p) = \{ v \in T_p M ~|~ B(p,v) = k \} \] 
of all points with Brillouin index $k$.
\end{defn}

In a purely mathematical setting, these have been studied in~\cite{bieberbach} and~\cite{jones} for lattices in Euclidean space and in~\cite{veerman} for discrete sets in certain metric spaces. In~\cite{skri}, the Brillouin zones are studied from an analytic number theoretic point of view. We refer to~\cite{pei2} for an overview of the interrelationships between the focal decomposition and physics, arithmetic and geometry. In~\cite{pugh, pugh_2}, the notion of focal stability was introduced, which is a local notion, where the results imply that in dimension two, in the absence of conjugate points, generically in the strong $C^{\infty}$-Whitney topology, the focal decomposition is locally topologically stable.

We are interested in the general question to what extent the information encoded in the focal decomposition determines the geometry of the underlying manifold. That is, we consider the global counterpart to focal stability. It has been shown in~\cite{KMP} that flat $n$-tori, with $n \geq 2$, are focally rigid, in the sense that global topological deformations of the focal decomposition are not possible without essentially changing the metric. Two surfaces $M_1$ and $M_2$ are commensurable, if their uniformizing surface groups $\Gamma_1$ and $\Gamma_2$ are commensurable, 
that is, if $\Gamma_1 \cap \Gamma_2$ has finite index in both $\Gamma_1$ and $\Gamma_2$.

\begin{thma}
Two closed hyperbolic surfaces that are focally equivalent are commensurable.
\end{thma}

Since a generic closed hyperbolic surface (of genus $g \geq 3$) is maximal by~\cite{green}, in the sense that it is not contained in a larger Fuchsian group, we have the following.

\begin{corb}
Generically, closed hyperbolic surfaces (in genus $g \geq 3$) which are focally equivalent, are isometric.
\end{corb}

The proof of Theorem A uses only an index-preserving homeomorphism between a pair of tangent planes, rather than the whole tangent bundle as the general definition stipulates. We believe the main result remains true if commensurable is replaced by isometric, by using this ambient structure.
In the setting of hyperbolic surfaces, the Brillouin zones are intimately connected with lattice counting estimates of the Fuchsian group corresponding to the surface. The fact that the surface is compact allows for uniform lattice counting estimates that gives rise to universal behaviour of the 
geometry of Brillouin zones. Theorem A in dimension $n \geq 3$ follows from the strong rigidity theorem by Mostow-Prasad~\cite{mostow, prasad}, according to which a closed hyperbolic $n$-manifold, with $n \geq 3$, is determined up to isometry by its fundamental group, and thus by its 
topology. We further refer to a result by Wolpert~\cite{wolp} that generically the length spectrum determines a closed hyperbolic surface up to isometry, and the counterexamples by Vign\'eras~\cite{vig} whom constructed commensurable closed hyperbolic surfaces for which this result fails.
See also the survey~\cite{kleiner} for more on rigidity in the setting of negatively curved manifolds. 

\section{Focal rigidity of hyperbolic surfaces}\label{sec_thm_A}

\subsection{Preliminaries and notation}\label{subsec_prelim}

A closed hyperbolic surface is a surface of the form $M = \D^2 \slash \Gamma$, where $\Gamma$ is a cocompact torsion-free Fuchsian group, with the metric induced by the Poincar\' e metric of the universal cover, denoted by $d(\cdot, \cdot)$. 
Since the exponential mapping $\exp_p \colon T_pM \ra M$ at the basepoint $p \in M$ is a covering mapping by the Cartan-Hadamard theorem, the exponential mapping is thus isomorphic, as a covering mapping, to the canonical covering 
mapping $\pi \colon \D^2 \ra \D^2 \slash \Gamma$. We adopt this identification in what follows. Assuming that two closed hyperbolic surfaces 
\begin{equation} 
M_1 = \D^2 \slash \Gamma_1, ~ \tu{and} ~ M_2 = \D^2 \slash \Gamma_2,
\end{equation}
with $\Gamma_1, \Gamma_2$ cocompact torsion-free Fuchsian groups, are focally equivalent, $M_1$ and $M_2$ are homeomorphic and there exists a homeomorphism $\varphi_0 \colon \D^2 \ra \D^2$ with the property that $\varphi_0(\sigma^1_i) = \sigma^2_i$ 
for $1 \leq i \leq \infty$, where we will only consider the homeomorphism $\varphi_0$ relative to two basepoints $p \in M_1$ and $q \in M_2$, which by conjugating $\Gamma_1$ and $\Gamma_2$ with a suitable M\"obius transformation, we may 
assume to correspond to $0 \in \D^2$ in the cover, so that $\varphi_0(0) = 0$, where $0 \in \D^2$. We will assume this normalization throughout the remainder. We denote $\varphi_0$ by $\varphi$, to simplify notation, and we henceforth suppress the reference to the basepoint $0$, that is, 
we write $\sigma_i(0) = \sigma_i$ and $B_k(0) = B_k$. Further, define
\begin{equation}
\Lambda_1 = \mathcal{O}_{\Gamma_1}(0) \subset \D^2, \quad \Lambda_2 = \mathcal{O}_{\Gamma_2}(0) \subset \D^2,
\end{equation} 
the orbits of the point $0 \in \D^2$ under $\Gamma_1$ and $\Gamma_2$ respectively. Observe that, since the groups $\Gamma_1$ and $\Gamma_2$ are torsion-free, there is a one-to-one correspondence between group elements in $\Gamma_i$ and 
lattice points in $\Lambda_i$ with $i=1,2$. In case no distinction has to be made between $\Gamma_1$ or $\Gamma_2$ or notions related to these, we suppress the index and denote the group $\Gamma$, its orbit $\Lambda = \mathcal{O}_{\Gamma}(0)$ 
and sets $B_k$, $k \in \N$, associated to the Brillouin zones, to simplify notation.

In order to prove Theorem A, we show the lattices $\Lambda_1$ and $\Lambda_2$ coincide up to rotation, for focally equivalent surfaces $M_1$ and $M_2$ in the above notation. This is sufficient to conclude that the groups $\Gamma_1$ and $\Gamma_2$ are commensurable. 

\subsection{Focal decomposition of a hyperbolic surface}

We start by constructing the focal decomposition of a closed hyperbolic surface.

\begin{defn}\label{bril_plane}
Define the geodesic $L_{\lambda} \subset \D^2$, $\lambda \in \Lambda$, by 
\begin{equation} 
L_{\lambda} = \{z \in \D^2 ~ \vert ~ d(z,0) = d(z,\lambda) \}~\tu{and}~ \LL = \bigcup_{\lambda \in \Lambda} L_{\lambda},
\end{equation}
where $\LL$ is referred to as the {\em web of geodesics}.
\end{defn}

\begin{rem}\label{rem_loc_finite}
The geodesic $L_{\lambda}$ is the perpendicular bisector of the geodesic segment joining $0, \lambda \in \D^2$. Further, as the orbit $\Lambda$ is discrete, the web $\LL$ is {\em locally finite} in the sense that every compact disk in 
$\D^2$ meets only finitely many distinct geodesics in $\LL$.
\end{rem}

Given $z \in \D^2$, define $\rho(z) \subset \D^2$ be the open geodesic ray connecting $0$ and $z$ in $\D^2$. Further, in what follows, given a point $z \in \D^2$, denote $r(z) = d(0,z)$. Define 
\begin{eqnarray}\label{eq_defn_i}
\iota(z) & = & \# \left\{ \lambda \in \Lambda ~\vert~ L_{\lambda} \cap \rho(z) \neq \emptyset \right\} \\
\upsilon(z) & = & \# \left\{ \lambda \in \Lambda ~\vert~ L_{\lambda} \ni z \right\}
\end{eqnarray}
where $\#$ denotes the cardinality of the set adjoined. In terms of these indices, the focal decomposition can be expressed as follows, see also~\cite{pei2} and~\cite{veerman}.

\begin{lem}[Focal decomposition of a hyperbolic surface]\label{lem_focal_hyperbolic}
We have 
\begin{equation}\label{lem_sigma}
\sigma_i = \left\{ z \in \D^2 ~|~ \upsilon(z) = i - 1 \right\}~\tu{and}~\D^2 = \sigma_1 \cup \LL,
\end{equation}
and 
\begin{equation}\label{lem_sigma_2}
\iota(z) = \# \left\{ \lambda \in \Lambda ~\vert~ \lambda \in D(z,r(z)) \right\}.
\end{equation}
\end{lem}

\begin{proof}
Take $q \in M$, take $z \in \pi^{-1}(q) \setminus \{ 0\}$ and consider $C(z, r(z))$. For every $\lambda \in \Lambda \cap C(z, r(z))$, the line $L_{\lambda}$ passes exactly through $z \in \D^2$. Indeed, consider the triangle formed by the 
vertices $0, \lambda, z \in \D^2$. If $\lambda \in C(z, r(z))$, then $d(0,z) = d(z, \lambda)$ and thus the triangle is isosceles. Further, the bisector $L_{\lambda}$ cuts the base of the triangle in half and crosses it perpendicularly. 
By symmetry, we must therefore have that $L_{\lambda} \ni z$. The distance from $z$ to any point of $\Lambda \cap C(z, r(z))$ is precisely $r(z)$ by construction and every geodesic arc connecting $z$ and an orbit point $\lambda \in \Lambda \cap C(z, r(z))$
projects to a geodesic curve of length $r(z)$ in $M$ that connects $p$ and $q$. If $z \in \sigma_i$, then there are exactly $i$ of these geodesic curves, and it thus follows that $i = \# \{ \Lambda \cap C(z, r(z)) \} = \upsilon(z) + 1$. This yields~\eqref{lem_sigma}.
Similarly, we have $L_{\lambda} \cap \rho(z) \neq \emptyset$ if and only if $\lambda \in D(z, r(z))$ from which~\eqref{lem_sigma_2} follows.
\end{proof}

\subsection{Topology and geometry of Brillouin zones}

The Brillouin zones exhibit interesting universal behaviour, which is to a large extent independent of the cocompact Fuchsian group $\Gamma$, 
and is a consequence of the uniform lattice counting estimates that hold for cocompact Fuchsian groups~\cite{buser}.
Combining these lattice counting estimates with~\eqref{lem_sigma_2} of Lemma~\ref{lem_focal_hyperbolic}, we obtain the following.

\begin{lem}\label{lem_lattice_count}
Given a cocompact Fuchsian group $\Gamma$ and $z \in \D^2$, we have
\begin{equation}\label{eq_lattice_est}
\left| \frac{\iota(z)}{\cosh^2( r(z)/2 )} - \frac{4 \pi}{ \area(F)} \right| = \OO \left( e^{-\alpha r(z)} \right).
\end{equation}
The constant $\alpha = \alpha_{\Gamma} > 0$ and the implied constant depend only on $\Gamma$. In particular, $\iota(z) \ra \infty$ as $r(z) \ra \infty$. 
\end{lem}

Topological properties of the Brillouin zones are summarized in the following lemma, cf.~\cite{veerman}.

\begin{lem}[Topology of Brillouin zones]\label{lem_brill_zones}
We have that
\begin{equation}
\Int(B_k) = \{ z \in \D^2 ~|~ \iota(z) = k ,~\upsilon(z) = 0\},
\end{equation} 
and $B_k = \Cl(\Int(B_k))$. In particular, the union of the Brillouin zones $\Int(B_k)$, $k \in \N$, is dense in $\D^2$.
\end{lem}

\begin{proof}
By Lemma~\ref{lem_focal_hyperbolic}, we have that $\iota(z) = k$ and $\upsilon(z)=0$ if and only if 
\begin{equation}
\# \{ \Lambda \cap D(z, r(z) \} = k,~\tu{and}~\# \{ \Lambda \cap C(z, r(z)) \}= \emptyset. 
\end{equation}
If $\upsilon(z) \neq 0$, then the ray $\rho(w)$, with $w =(1+\epsilon)z$, will cross additional geodesics in the web $\LL$, and thus $\iota(w) > k$. Therefore, points for which $\upsilon(z) > 0$ are not in the interior of any $B_k$. Conversely, the condition that $\upsilon(z) = 0$ is clearly open.

To prove that $B_k$ is compact, we define the function that associates the maximal distance from the origin to $B_k$ as a function of the angle of the ray with respect to the positive real axis. 
As $\iota(z) \ra \infty$ as $r(z) \ra \infty$ by Lemma~\ref{lem_lattice_count}, this function associates a definite real number to every angle. Further, it is continuous as a function of the angle, since the web is locally finite 
and every geodesic in the web $\LL$ that intersects a ray emanating from the origin crosses the ray transversely. Since the circle is compact, continuity yields a finite maximum over all angles. Thus $B_k$ is bounded.
It follows that $B_k$ is compact as it is also closed. Since the web of geodesics $\LL$ is locally finite, cf. Remark~\ref{rem_loc_finite}, by compactness of $B_k$, $\partial B_k$ consists of a finite union of geodesic arcs bounding finitely many geodesic polygons and it 
follows that $B_k = \Cl(\Int(B_k))$. To prove density, the web of geodesics $\LL \subset \D^2$, being a locally finite union of geodesics, is nowhere dense. Thus $\sigma_1 = \D^2 \setminus \LL = \bigcup_{k \in \N} \Int(B_k)$ is dense in $\D^2$.
\end{proof}

\begin{rem}
The first Brillouin zone appears in several different contexts; for example, it appears as the \emph{Wigner-Seitz cell} in physics, the  \emph{Voronoi cell} in the study of circle packings and, specific to our setting, as the interior of the \emph{Dirichlet region} 
of a Fuchsian group. 
\end{rem}

The following result expresses the universal geometrical behaviour of the Brillouin zones.

\begin{lem}[Geometry of Brillouin zones]\label{lem_geo_brill}
Let $\Gamma$ be a cocompact Fuchsian group with $M = \D^2 \slash \Gamma$ a closed surface of genus $g$. There exists a decreasing function $\epsilon(t)$, with $t \in [0, \infty)$, 
depending only on $\Gamma$ and $k \in \N$, such that if $z \in B_k \subset \D^2$, then
\begin{equation}\label{eq_dist_Brillouin}
\tau(g,k) - \epsilon(r(z)) \leq r(z) \leq \tau(g,k) + \epsilon(r(z)),
\end{equation}
where $\tau(g, k) = \log(4(g-1)k)$. Furthermore, $\epsilon(t) \ra 0$ as $t \ra \infty$. 
\end{lem}

\begin{proof}
The hyperbolic area of a fundamental domain $F$ of $\Gamma$ is given by $\area(F) = 4 \pi(g-1)$ by Gauss-Bonnet, and we denote
\[ \beta(g) = \frac{4 \pi}{ \area(F)} = \frac{1}{g-1}. \]
We can rewrite~\eqref{eq_lattice_est} in Lemma~\ref{lem_lattice_count} as 
\begin{equation}
\beta - \widehat{C}(r(z)) \leq \frac{\iota(z)}{\cosh^2( r(z)/2 )} \leq \beta + \widehat{C}(r(z)),
\end{equation}
where $\widehat{C}(t) \ra 0$ for $t \ra \infty$ and is decreasing. Equivalently, 
\begin{equation}
\frac{\iota(z)}{\beta + \widehat{C}(r(z))} \leq \cosh^2( r(z)/2 ) \leq \frac{\iota(z)}{\beta - \widehat{C}(r(z))}.
\end{equation}
As $\cosh^2(t/2) = \frac{1}{4}(e^{t} + e^{-t} + 2)$, taking logarithms, we find that 
\begin{equation}\label{eq_lattice_est_2}
\log \left( \frac{4\iota(z)}{\beta + \widehat{C}(r(z))} \right) \leq \log \left( e^{r(z)} + e^{-r(z)} + 2 \right) \leq \log \left( \frac{4\iota(z)}{\beta - \widehat{C}(r(z))} \right), 
\end{equation}
for all $z \in \D^2$. Define the function $\kappa \colon [0, \infty) \ra (0, \log(4)]$ by 
\begin{equation}\label{eq_defn_g}
\kappa(t) = \log(e^t + e^{-t} +2) - t = \log(1 + e^{-2t} + 2e^{-t}).
\end{equation} 
It follows that $\kappa(t) \ra 0$ as $t \ra \infty$. From~\eqref{eq_lattice_est_2} and~\eqref{eq_defn_g}, we find
\begin{equation}\label{eq_lattice_est_3}
\kappa(r(z)) + \log \left( \frac{4\iota(z)}{\beta + \widehat{C}(r(z))} \right) \leq r(z) \leq \log \left( \frac{4\iota(z)}{\beta - \widehat{C}(r(z))} \right) + \kappa(r(z)).
\end{equation}
Since $\widehat{C}(r(z)) \ra 0$ for $r(z) \ra \infty$, we can find a function $\epsilon(t)$, with $\epsilon(t) \ra 0$ as $t \ra \infty$, such that 
\begin{equation}\label{eq_lattice_est_4}
\log \left( \frac{4\iota(z)}{\beta} \right) - \epsilon(r(z)) \leq r(z) \leq \log \left( \frac{4\iota(z)}{\beta} \right)  + \epsilon(r(z)).
\end{equation}
Setting $\tau(g,k) := \log(4(g-1)k)$ yields~\eqref{eq_dist_Brillouin} for $z \in \Int(B_k)$. By continuity, the same estimates hold for $B_k$, since $B_k$ is the closure of $\Int(B_k)$.
\end{proof}

Further, the homeomorphism $\varphi$ acts in a coherent way on the individual geodesics in the web $\LL$, in the following sense.

\begin{lem}\label{lem_lines_to_lines}
Given $\lambda_1 \in \Lambda_1$, there exists a unique $\lambda_2  \in \Lambda_2$ such that $\varphi(L_{\lambda_1}) = L_{\lambda_2}$.
\end{lem}

\begin{proof}
Given a point $z \in L_{\lambda_1}$ with $\upsilon(z) = 1$, there exists a unique $L_{\lambda_2}$ that passes through $w = \varphi(z)$. We show that $\varphi(L_{\lambda_1}) \subseteq L_{\lambda_2}$. 
A similar argument shows that $\varphi^{-1}(L_{\lambda_2}) \subseteq L_{\lambda_1}$. Now suppose that $z \in L_{\lambda_1}$ is a point through which $\upsilon(z) \geq 2$ geodesics in the web $\LL_1$ pass, 
take a small Euclidean disk centered at $z$, which does not intersect geodesics other than those that pass through $z$. The geodesic $L_{\lambda_1}$ cuts this disk into two halves, each containing the same number 
of segments of geodesics incident to $z$. Since this is a topological invariant, the curve $\varphi(L_{\lambda_1})$ has to pass through as the same geodesic at $w = \varphi(z)$. As this holds for every such intersection point, the claim follows.
\end{proof}

Every geodesic $L_{\lambda}$, with $\lambda \in \Lambda$, separates the disk $\D^2$ into two connected components $\HH^{\pm}_{\lambda}$, where we denote $\HH^-_{\lambda}$ the component containing $0 \in \D^2$. 
We say $L_{\lambda}$ {\em separates} $0, z \in \D^2$ if $z \in \HH^+_{\lambda}$. The Brillouin zones are natural with respect to our homeomorphism $\varphi$ in the following sense.

\begin{lem}[Naturality of Brillouin zones]\label{lem_zones_to_zones}
We have $\varphi(B^1_k) = B^2_k$, for every $k \in \N$.
\end{lem}

\begin{proof}
First, let $z \in \Int(B^1_k)$. By Lemma~\ref{lem_focal_hyperbolic}, we have that
\[ \iota(z) = \# \left\{ \lambda \in \Lambda_1 ~\vert~ L_{\lambda} \cap \rho(z) \neq \emptyset \right\} = k . \]
Every such geodesic $L_{\lambda}$ separates $0$ and $z$ and there are exactly $k$ such geodesics in the web $\LL_1$ relative to $\Lambda_1$.
As $\varphi$ is a homeomorphism for which $\varphi(0) = 0$ and $\varphi(L_{\lambda_1}) = L_{\lambda_2}$, with $\lambda_1 \in \Lambda_1$ and $\lambda_2 \in \Lambda_2$, by Lemma~\ref{lem_lines_to_lines}, 
this information is preserved by $\varphi$. Thus there exist exactly $k$ geodesics in the web $\LL_2$ relative to $\Lambda_2$ that separate $0$ and $w = \varphi(z)$. This shows that $\varphi( \Int(B^1_k) ) = \Int(B^2_k)$, for every $k \in \N$. 
Passing to the closure yields $\varphi(B^1_k) = B^2_k$, for every $k \in \N$.
\end{proof}

From the above, we obtain the following.

\begin{lem}[Radial quasi-isometry]\label{lem_quasi_iso}
There exists a function $C(t)$, such that
\begin{equation}\label{eq_rad_quasi_1}
r(z) - C(r(z)) \leq r( \varphi(z) ) \leq r(z) + C(r(z)),
\end{equation}
for every $z \in \D^2$, with $C(t) \ra 0$ for $t \ra \infty$.
\end{lem}


\begin{proof}
Given $z \in \D^2$, $z \in B^1_k$ for some $k \geq 1$. By Lemma~\ref{lem_zones_to_zones}, we have that $w := \varphi(z) \in B^2_k \subset \D^2$. Since the genus of $M_1$ and $M_2$ is equal, by Lemma~\ref{lem_geo_brill}, we have that 
\begin{equation}\label{eq_est_rad_1}
\tau(g,k) - \epsilon_1(r(z)) \leq r(z) \leq \tau(g,k) + \epsilon_1(r(z))
\end{equation}
and 
\begin{equation}\label{eq_est_rad_2}
\tau(g,k) - \epsilon_2(r(w)) \leq r(w) \leq \tau(g,k) + \epsilon_2(r(w)),
\end{equation}
with $\tau(g,k) = \log(4(g-1)k)$ and where $\epsilon_1(t), \epsilon_2(t) \ra 0$ as $t \ra \infty$. In particular, $\epsilon_1(t)$ and $\epsilon_2(t)$ are bounded for all $t \in [0, \infty)$. 
Combining~\eqref{eq_est_rad_1} and~\eqref{eq_est_rad_2}, there exists a constant $r_0 > 0$ such that $r(w) \geq r(z) - r_0$, and we have 
\begin{equation}\label{eq_rad_quasi_2}
| r(w) - r(z) | \leq \epsilon_1(r(z)) + \epsilon_2(r(w)) \leq \epsilon_1(r(z)) + \epsilon_2(r(z) - r_0) := C(r(z)).
\end{equation}
Thus~\eqref{eq_rad_quasi_1} follows from~\eqref{eq_rad_quasi_2} and $C(t) \ra 0$ as $t \ra \infty$, since $\epsilon_1(t), \epsilon_2(t) \ra 0$ as $t \ra \infty$.
\end{proof}

\subsection{The induced mapping at infinity}

Given an interval $I \in \SSS^1$, we denote $|I|$ the length of the interval, relative to the standard Euclidean measure on $\SSS^1$. Given $\lambda \in \Lambda$, we denote $I_{\lambda} \subset \SSS^1$ the shortest closed interval whose endpoints correspond to the endpoints of $L_{\lambda}$ on $\SSS^1$. 

Given a geodesic $L_{\lambda} \subset \LL$, denote $\delta(\lambda) = d(0, L_{\lambda})$. We denote $\II := \{ I_{\lambda} \}_{\lambda \in \Lambda}$ the collection of these closed intervals associated to the web $\LL$. We first collect combinatorial information about the collection of intervals $\II$.

\begin{lem}\label{properties_I}
The collection of intervals $\II$ satisfies the following conditions.
\begin{enumerate}
\item[\tu{(a)}] Every point $x \in \SSS^1$ is the limit point of an infinite nested sequence of intervals in the collection $\II$.
\item[\tu{(b)}] Given an interval $J \subset \SSS^1$ and given $\epsilon >0$, there exists a finite covering of $J$ by intervals in $\II$ of length at most $\epsilon$.
\end{enumerate}
\end{lem}

\begin{proof}
To prove (a), we first observe that the lengths of the intervals in $\II$ form a null-sequence; that is, for any given $\epsilon > 0$, there are only finitely many intervals in $\II$ whose lengths exceed $\epsilon$. Indeed, if there would be infinitely 
many intervals in $\II$ whose length is bounded from below, then the distance from the origin of the geodesics in the web $\LL$ that correspond to this infinite collection of intervals is uniformly bounded from above. 
However, this contradicts that the web $\LL$ is locally finite. As $r(z) \ra \infty$, we have that $\iota(z) \ra \infty$. By Lemma~\ref{lem_brill_zones}, the number of elements in $\Lambda$ for which $L_{\lambda} \cap \rho(z) \neq \emptyset$ 
increases as $r(z) \ra \infty$. Thus, given $x \in \mathbb{S}^1$, we can find infinitely many intervals that cover $x \in \mathbb{S}^1$. Passing to a subsequence if necessary to guarantee nesting, this proves claim (a).

To prove (b), let $J \subset \SSS^1$ be any given interval and $\epsilon>0$. By deleting finitely many geodesics $L_{\lambda}$ from $\D^2$, and corresponding intervals $I_{\lambda}$ from $\mathbb{S}^1$, according to their increasing distance $\delta(\lambda)$, 
we may assume that all intervals in the remaining collection are of length at most $\epsilon$. We obtain a convex hull defined as the connected component, containing the origin $0 \in \D^2$, of the complement of $\LL$ minus the finitely many geodesics just deleted.
To finish the proof, we argue as in Lemma~\ref{lem_brill_zones}. The convex hull consists of finitely many edges. Indeed, for every ray emanating from the origin, define the function that associates the distance to this convex hull as a function of the angle of the ray with respect to the positive real axis. By (a), this function associates a definite real number to every angle and is continuous. By compactness of the circle, continuity yields a finite maximum over all angles. Moreover, the convex hull is comprised of finitely many edges since the web of geodesics $\LL$ is locally finite.
The edges of this convex hull can be continued to complete geodesics contained in $\LL$. The intervals in $\SSS^1$ corresponding to this finite collection of geodesics, by construction, gives a finite covering of $\SSS^1$ by intervals in $\II$ whose lengths are at most $\epsilon$. In particular, it gives a finite covering of $J \subset \SSS^1$ by such intervals.
\end{proof}

\begin{lem}\label{lem_extend}
The homeomorphism $\varphi \colon \D^2 \ra \D^2$ extends to a homeomorphism $f \colon \mathbb{S}^1 \ra \mathbb{S}^1$ of the boundary $\mathbb{S}^1 = \partial \D^2$.
\end{lem}

\begin{proof}
Given $\lambda_1 \in \Lambda_1$, let $\lambda_2 \in \Lambda_2$ defined by $\varphi(L_{\lambda_1}) = L_{\lambda_2}$. 
As $\varphi(0) = 0$, we have that $\varphi(\HH^+_{\lambda_1}) = \HH^+_{\lambda_2}$, thus it follows that $f(I_{\lambda_1}) = I_{\lambda_2}$,
and the endpoints of the interval $I_{\lambda_1} \in \II_1$ are sent to the endpoints of $I_{\lambda_2} \in \II_2$.

Given a point $x \in \mathbb{S}^1$, by Lemma~\ref{properties_I} (i), $\{ x \} = \bigcap_k I_{\lambda^k_1}$ for some subsequence $(\lambda^k_1)_{k \in \N}$ of nested intervals decreasing in length to zero. As this nesting is preserved by $f$, due to the nesting property 
of the corresponding half-planes $\HH^+_{\lambda^k_1}$ and $\HH^+_{\lambda^k_2}$ preserved by $\varphi$ in $\D^2$, we have that $\{ y \} := \bigcap_k I_{\lambda^k_2}$ is a unique point in the image and $f(x) = y$. Therefore, $f$ is one-to-one. 
To prove continuity of $f$, given $x \in \mathbb{S}$ and $y=f(x)$ and an interval $\widehat{J} \subset \mathbb{S}^1$ containing $y$, again by Lemma~\ref{properties_I} (i), we can find an interval $I_{\lambda_2} \subset \widehat{J}$ containing $y$. Consequently, 
$J \subset I_{\lambda_1}$, where $I_{\lambda_1} = f^{-1}(I_{\lambda_2})$ containing $x$, has the property that $f(J) \subset I_{\lambda_2} \subset \widehat{J}$. Thus $f$ is a homeomorphism.
\end{proof}

\begin{lem}\label{lem_unif_quotient}
There exist uniform constants $1 \leq K(t) \leq K_0$ for all $t \in [0, \infty)$, such that if $L_{\lambda_2} = \varphi(L_{\lambda_1})$, then 
\begin{equation}\label{eq_quotient_intervals}
\frac{1}{K(\delta(\lambda_1))} \leq \frac{| I_{\lambda_2} | }{| I_{\lambda_1} | } \leq K(\delta(\lambda_1)),
\end{equation}
where $K(t) \ra 1$ for $t \ra \infty$.
\end{lem}

\begin{proof}
If we are given $L_{\lambda_1}$ and $L_{\lambda_2} = \varphi(L_{\lambda_1})$, then by Lemma~\ref{lem_quasi_iso} we have that 
\begin{equation*}
\delta(\lambda_2) = \delta(\lambda_1) + \zeta(\lambda_1),
\end{equation*}
with $| \zeta(\lambda_1) | \leq C(\delta(\lambda_1))$. The Euclidean distance $\bar{\delta}(\lambda)$ of $0 \in \D^2$ to $L_{\lambda}$ relates to the hyperbolic distance $\delta(\lambda)$ as 
\[ \bar{\delta}(\lambda) = \tanh(\delta(\lambda)/2).\] 
The length $|I_{\lambda}|$ of the interval $I_{\lambda}$ is estimated as follows. We observe that for $\delta(\lambda) \ra \infty$, the geodesic $L_{\lambda}$ converges, when rescaling to unit size, to a semicircle whose distance from the origin is given by $\delta(\lambda)$.
Since the Euclidean radius of the semicircle is $1- \bar{\delta}(\lambda)$, there exists a function $\eta(t)$, with $\eta(t) \ra 0$ as $t \ra 0$, such that
\begin{equation*}
| I_{\lambda} | = 2(1- \bar{\delta}(\lambda))( 1 + \eta(|I_{\lambda} |)),
\end{equation*}
and thus
\begin{equation}\label{eq_semi_circle}
| I_{\lambda} | = 2 \left( 1 - \tanh(\delta(\lambda)/2) \right) ( 1+ \eta(| I_{\lambda} |)).
\end{equation} 
Denoting $\delta_1 = \delta(\lambda_1)$, 
$\delta_2 = \delta(\lambda_2)$ and $\zeta_1 = \zeta(\lambda_1)$ for brevity, combining~\eqref{eq_semi_circle} with the estimate
\begin{equation*}
\frac{1+ \eta(| I_{\lambda_2} |)}{1+ \eta(| I_{\lambda_1} |)} \ra 1,
\end{equation*}
as $| I_{\lambda_1} |, | I_{\lambda_2} | \ra 0$, we have that 
\begin{equation}
\frac{| I_{\lambda_2} | }{| I_{\lambda_1} | } \asymp \frac{ 1 - \tanh(\delta_2/2)} {1 - \tanh(\delta_1/2)} = \frac{1 + e^{\delta_1} }{1+ e^{\delta_2} } = \frac{1+e^{\delta_1}}{1+ e^{\delta_1 + \zeta_1}} \asymp \frac{e^{\delta_1}}{e^{\delta_1 + \zeta_1}} = e^{- \zeta_1} \ra 1,
\end{equation}
since $\delta_1 \ra \infty$ and, simultaneously, $|\zeta_1| \leq C(\delta_1) \ra 0$ as $\delta_1 \ra \infty$. Thus~\eqref{eq_quotient_intervals} follows.
\end{proof}

This implies that the boundary homeomorphism $\varphi$ acts trivially, in the following sense.


\begin{lem}\label{lem_circle_homeo}
The circle homeomorphism $f \colon \mathbb{S}^1 \ra \mathbb{S}^1$ induced by $\varphi$ is a rotation.
\end{lem}

\begin{proof}
We first observe it is sufficient to show that the homeomorphism $f$ is Lipschitz for some Lipschitz constant $1 \leq K_f < \infty$. Suppose that this is proved. 
It then follows that $f$ is differentiable almost everywhere and that $|f(x)-f(y)| = \int_x^y Df$. By Lemma~\ref{properties_I} (a), every point $x \in \SSS^1$ is a limit point of a nested sequence of intervals of $\II_1$. By Lemma~\ref{lem_unif_quotient}, 
the ratio of the lengths of the nested intervals in $\II_1$ converging to $x$ and the corresponding image intervals in $\II_2$ converging to $f(x)$ converges to $1$. This yields that $Df(x) =1$ at every point $x \in \mathbb{S}^1$ where $f$ is differentiable. 
It thus follows that $Df(x) =1$ for almost every $x \in \mathbb{S}^1$, and integration then yields that $|f(x)-f(y)| = |x-y|$, so $f$ has to be a rotation.

To prove that $f$ is Lipschitz, we show that $f$ does not shrink an interval by a factor exceeding $(4K_0)^{-1}$, where $K_0$ is the uniform constant of Lemma~\ref{lem_unif_quotient}. A similar argument for the inverse $f^{-1}$ shows that $f$ does not expand intervals by the same factor either. 
Take any two points $x, y \in \SSS^1$ sufficiently close, say $\ell = |x-y| < 1/4$, and denote $J$ the (shortest) interval with endpoints $x$ and $y$. By Lemma~\ref{properties_I} (b), we can find a finite covering of $J$ by a collection of intervals $\{ I_{\lambda_k} \}_{k=1}^N$ whose lengths are at most $\ell /10$. By deleting a finite number of intervals of this covering, we may assume this covering is minimal in the sense that deleting any one more interval of the collection would make it fail to be a covering. 

First, by minimality, it is readily shown that the left-endpoints of the intervals are distinct, so we can label the intervals according to the ordering of the left-endpoints $x_k$ of $I_{\lambda_k}$ in $\SSS^1$. Further, by minimality, it is shown that the odd-labeled intervals are mutually disjoint, 
and so with the even-labeled intervals. Define $\Sigma_{\tu{odd}}$ and $\Sigma_{\tu{even}}$ to be the sum of the lengths of the odd- and even-labeled intervals respectively. Either $\Sigma_{\tu{odd}}$ or $\Sigma_{\tu{even}}$ has to be of length at least $\ell/2$. Indeed, if both 
$\Sigma_{\tu{odd}}$ and $\Sigma_{\tu{even}}$ would be strictly less than $\ell/2$, then their union could not cover an interval of length $\ell$, a contradiction. 

Suppose that $\Sigma_{\tu{odd}} \geq \ell/2$. As $J$ contains all odd-labeled intervals, except possibly parts of the first and last, the sum of the lengths of the remaining odd-labeled intervals strictly contained in $J$ is at least $\ell / 2  - 2\ell/10 > \ell/4$.
Similarly, in case $\Sigma_{\tu{even}} \geq \ell/2$, the sum of the lengths of all even-labeled intervals strictly contained in $J$ is larger than $\ell/4$. By Lemma~\ref{lem_unif_quotient}, the lengths of these intervals are not shrunk by a factor more than $K_0$ by $f$, 
and thus $f$ can not shrink $J$ by a factor exceeding $(4K_0)^{-1}$, as required. 
\end{proof}

\subsection{Proof of the main result}

We now finish the proof of Theorem A. By Lemma~\ref{lem_circle_homeo}, $\varphi \colon \D^2 \ra \D^2$ extends to $\mathbb{S}^1 = \partial \D^2$ as a rotation. 
By conjugating $\Gamma_2$ with the corresponding rotation fixing the origin $0 \in \D^2$, we may assume that the induced action on the boundary 
is the identity. In that case we have that $\Lambda_1 = \Lambda_2$. Indeed, let $L_{\lambda_1} \subset \D^2$ be a geodesic with endpoints $x_1 ,x_2 \in \mathbb{S}^1$, with $\lambda_1 \in \Lambda_1$. As $f(x_i) = x_i$, for $i=1,2$, and $L_{\lambda_2} = \varphi(L_{\lambda_1})$, we must have that $L_{\lambda_1} = L_{\lambda_2}$, as a geodesic is uniquely determined by its endpoints. Since $L_{\lambda}$ is the perpendicular bisector of $0$ and $\lambda$ in $\D^2$, it follows that $\lambda_1 = \lambda_2 \in \Lambda_2$. Since this holds for every point in $\Lambda_1$, it follows that $\Lambda_1 = \Lambda_2$. Therefore, we have that $\Lambda:= \Lambda_1 = \Lambda_2$. 

To finish the proof, we need to show this condition implies the surfaces $M_1$ and $M_2$ are commensurable. Consider the stabilizer $\Gamma_{\Lambda} := \{ \gamma \in \Mob(\D^2) ~|~ \gamma( \Lambda ) = \Lambda \}$ of the discrete set $\Lambda$. 
Since $\Gamma_{\Lambda}$ contains $\Gamma_1$ and $\Gamma_2$ as subgroups, $\Gamma_{\Lambda} $ is not elementary. Either $\Gamma_{\Lambda} $ acts properly discontinuously on $\D^2$, or else $\Gamma_{\Lambda}$ contains an elliptic element of infinite order. 
However, since $\Lambda$ is a discrete and $\Gamma_{\Lambda}$-invariant set, the latter is impossible. Therefore, $\Gamma_{\Lambda}$ acts properly discontinuously and is thus Fuchsian. Consequently, since $\Gamma_1$ and $\Gamma_2$ act cocompactly, the index of $\Gamma_1$ 
and $\Gamma_2$ in $\Gamma_{\Lambda}$ is finite. Since the index of $\Gamma_1 \cap \Gamma_2$ in $\Gamma_1$ and $\Gamma_2$ is bounded by the index of $\Gamma_1$ and $\Gamma_2$ in $\Gamma_{\Lambda}$ respectively by standard group theory, it follows that 
$\Gamma_1 \cap \Gamma_2$ has finite index in both $\Gamma_1$ and $\Gamma_2$. That is, $\Gamma_1$ and $\Gamma_2$ are commensurable. This proves Theorem A.

To prove Corollary B, for a generic $\Gamma_1$, the stabilizer $\Gamma_{\Lambda}$ in the above notation equals $\Gamma_1$ by~\cite{green}. Since the orbit of $0$ under $\Gamma_2$ equals that of $\Gamma_1$, if $\Gamma_1 \cap \Gamma_2 \neq \Gamma_1$, 
we can take an element $\mu \in \Gamma_2$ not contained in $\Gamma_1$ and take the group generated by $\Gamma_1$ and $\mu$. Since this group is contained in the stabilizer $\Gamma_{\Lambda}$ and properly extends $\Gamma_1$, we have a contradiction. 
It follows that $\Gamma_1 = \Gamma_2$.

\section{Further Remarks}\label{sec_remarks}

We finish with several remarks regarding the results presented.

\subsection{Focal spectrum of hyperbolic surfaces}
  
Let us first pose a problem regarding the focal spectrum of a closed hyperbolic surface, see also~\cite{KMP} for the case of flat tori. An analogue of the length spectrum is given by the focal spectrum defined in terms of the focal decomposition. Given a basepoint $p \in M = \D^2 \slash \Gamma$ 
in the surface, assume that the lift corresponds to $0 \in \D^2$, denote $\Lambda = \OO_{\Gamma}(0)$ and let $\{ L_{\lambda} \}_{\lambda \in \Lambda}$ be the collection of Brillouin lines relative to the data as given. Now record the radii $r$, with multiplicity, for which the hyperbolic circle 
$C(0, r)$ in $\D^2$ meets either (i) a line $L_{\lambda}$ tangentially or (ii) the intersection of two or more such lines. We pose the problem as to whether this spectrum being equal relative to the basepoints $p \in M_1$ and $q \in M_2$ chosen is equivalent to the condition that there is an 
index-preserving homeomorphism $\varphi_0 \colon T_pM_1 \ra T_qM_2$.

\subsection{Generalizations of focal rigidity}

It would be interesting to determine to what extent the current results extend into the more general setting of variable negatively curved manifolds. Indeed, the covering map on the universal cover is isomorphic to the exponential mapping as before, which gives an explicit construction of the 
focal decomposition in terms of equidistant loci in the cover and associated Brillouin zones. Uniform lattice counting estimates subsequently give bounds on their shape in the cover, which has to be respected by the homeomorphism sending one structure to another. 
The Brillouin zones have a tendency to exhibit behavior that uniformizes the focal decomposition on a large scale and forms the driving force behind rigidity. It is an interesting problem to decide whether focally equivalent closed manifolds of non-positive 
curvature are commensurable, or isometric, modulo rescaling.

\subsection*{Acknowledgements}
The author thanks Mauricio Peixoto for many discussions and his valuable comments on the manuscript, and Caroline Series for several helpful discussions. 
Further, the author thanks the referee for comments and remarks on the manuscript. The author received financial support from CNPq (PCI) at IMPA.

\end{document}

%% file: macros.tex
\newtheorem{thm}{Theorem}
\newtheorem{cor}[thm]{Corollary}
\newtheorem{lem}[thm]{Lemma}
\newtheorem{prop}[thm]{Proposition}

\newtheorem*{thma}{Theorem A}
\newtheorem*{corb}{Corollary B}
\newtheorem*{strat}{Existence Theorem}
\newtheorem*{focal}{Focal Rigidity Conjecture}
\newtheorem*{mostow}{Mostow's Rigidity Theorem}
\newtheorem*{open}{Open problem}
\newtheorem*{conj}{Conjecture}

\newtheorem*{focalconj}{Focal Rigidity Conjecture}

\theoremstyle{remark}
\newtheorem*{notation}{Notation}

\theoremstyle{definition}
\newtheorem{defn}{Definition}
\newtheorem{rem}{Remark}
\newtheorem*{problem}{Problem}
\newtheorem{quest}{Question}

\numberwithin{equation}{section}
\newcommand{\figref}[1]{Figure~\ref{#1}}
\newcommand{\pichere}[2]
{\begin{center}\includegraphics[width=#1\textwidth]{#2}\end{center}}

\newcommand{\bignote}[1]{\begin{quote} \sf #1 \end{quote}}

\newcommand{\QED}{\rlap{$\sqcup$}$\sqcap$\smallskip}

\renewcommand{\Im}{\operatorname{Im}}
\renewcommand{\exp}{\operatorname{exp}}

\def\sss{\subsubsection}

\newcommand{\di}{\partial}
\newcommand{\dibar}{\bar\partial}
\newcommand{\hookra}{\hookrightarrow}
\newcommand{\ra}{\rightarrow}
\newcommand{\hra}{\hookrightarrow}
\newcommand{\imply}{\Rightarrow}
\def\lra{\longrightarrow}
\newcommand{\wc}{\underset{w}{\to}}
\newcommand{\tu}{\textup}
\newcommand{\SO}{\textup{SO}}
\newcommand{\Int}{\textup{Int}}
\newcommand{\Cl}{\textup{Cl}}

\newcommand{\eps}{{\varepsilon}}
\newcommand{\epsi}{{\epsilon}}
\newcommand{\veps}{{\varepsilon}}
\newcommand{\De}{{\Delta}}
\newcommand{\de}{{\delta}}
\newcommand{\la}{{\lambda}}
\newcommand{\La}{{\Lambda}}
\newcommand{\si}{{\sigma}}
\newcommand{\Si}{{\Sigma}}
\newcommand{\Om}{{\Omega}}
\newcommand{\om}{{\omega}}

\newcommand{\AAA}{{\mathcal A}}
\newcommand{\BB}{{\mathcal B}}
\newcommand{\CC}{{\mathcal C}}
\newcommand{\DD}{{\mathcal D}}
\newcommand{\EE}{{\mathcal E}}
\newcommand{\EEE}{{\mathcal O}}
\newcommand{\II}{{\mathcal I}}
\newcommand{\FF}{{\mathcal F}}
\newcommand{\GG}{{\mathcal G}}
\newcommand{\JJ}{{\mathcal J}}
\newcommand{\HH}{{\mathcal H}}
\newcommand{\KK}{{\mathcal K}}
\newcommand{\LL}{{\mathcal L}}
\newcommand{\MM}{{\mathcal M}}
\newcommand{\NN}{{\mathcal N}}
\newcommand{\OO}{{\mathcal O}}
\newcommand{\PP}{{\mathcal P}}
\newcommand{\QQ}{{\mathcal Q}}
\newcommand{\QM}{{\mathcal QM}}
\newcommand{\QP}{{\mathcal QP}}
\newcommand{\QL}{{\mathcal Q}}

\newcommand{\RR}{{\mathcal R}}
\renewcommand{\SS}{{\mathcal S}}
\newcommand{\TT}{{\mathcal T}}
\newcommand{\TTT}{{\mathcal P}}
\newcommand{\UU}{{\mathcal U}}
\newcommand{\VV}{{\mathcal V}}
\newcommand{\WW}{{\mathcal W}}
\newcommand{\XX}{{\mathcal X}}
\newcommand{\YY}{{\mathcal Y}}
\newcommand{\ZZ}{{\mathcal Z}}

\newcommand{\A}{{\mathbb A}}

\newcommand{\SSS}{{\mathbb S}}

\newcommand{\C}{{\mathbb C}}
\newcommand{\bC}{{\bar{\mathbb C}}}
\newcommand{\D}{{\mathbb D}}
\newcommand{\Hyp}{{\mathbb H}}
\newcommand{\J}{{\mathbb J}}
\newcommand{\Ll}{{\mathbb L}}
\renewcommand{\L}{{\mathbb L}}
\newcommand{\M}{{\mathbb M}}
\newcommand{\N}{{\mathbb N}}
\newcommand{\Q}{{\mathbb Q}}
\newcommand{\R}{{\mathbb R}}
\newcommand{\B}{{\mathbb B}}
\newcommand{\T}{{\mathbb T}}
\newcommand{\V}{{\mathbb V}}
\newcommand{\U}{{\mathbb U}}
\newcommand{\W}{{\mathbb W}}
\newcommand{\X}{{\mathbb X}}
\newcommand{\Z}{{\mathbb Z}}
\newcommand{\ball}{{\mathbb B}}
\newcommand{\Id}{\textup{Id}}
\newcommand{\cl}{\textup{Cl}}
\newcommand{\area}{\textup{area}}
\newcommand{\Mob}{\textup{M\"ob}}

\newcommand{\VVV}{{\mathbf U}}
\newcommand{\UUU}{{\mathbf U}}

\newcommand{\tT}{{\mathrm{T}}}
\newcommand{\tD}{{D}}
\newcommand{\hyp}{{\mathrm{hyp}}}

\newcommand{\f}{{\bf f}}
\newcommand{\g}{{\bf g}}
\newcommand{\h}{{\bf h}}
\renewcommand{\i}{{\bar i}}
\renewcommand{\j}{{\bar j}}

\catcode`\@=12

\def\Empty{}
\newcommand\oplabel[1]{
  \def\OpArg{#1} \ifx \OpArg\Empty {} \else
   \label{#1}
  \fi}

%

\long\def\realfig#1#2#3#4{
\begin{figure}[htbp]
\centerline{\psfig{figure=#2,width=#4}}
\caption[#1]{#3}
\oplabel{#1}
\end{figure}}

\numberwithin{figure}{section}

%

\newcommand{\comm}[1]{}
\newcommand{\comment}[1]{}